\documentclass{amsart}
\usepackage[english]{babel}
\usepackage{mathrsfs, mathtools, amssymb, mathabx, stmaryrd, mathdots}
\usepackage{dsfont}
\usepackage{tikz}
\usepackage{tikz-cd}
\usepackage[paper=a4paper, margin=2.8cm, bottom=49mm]{geometry}

\usepackage{hyperref}
\hypersetup{
    pdftoolbar=true,        
    pdfmenubar=true,        
    pdffitwindow=false,     
    pdfstartview={FitH},    
    pdftitle={},    
    pdfauthor={},     
    pdfkeywords={}, 
    pdfnewwindow=true,      
    colorlinks=true,       
    linkcolor=blue,          
    citecolor=green,        
    filecolor=magenta,      
    urlcolor=cyan           
}
\theoremstyle{plain}
\newtheorem{theorem}{Theorem}[section]
\newtheorem{lemma}[theorem]{Lemma}
\newtheorem{proposition}[theorem]{Proposition}
\newtheorem{corollary}[theorem]{Corollary}

\theoremstyle{definition}
\newtheorem{definition}[theorem]{Definition}
\newtheorem{remark}[theorem]{Remark}
\newtheorem{example}[theorem]{Example}
\newtheorem{question}[theorem]{Question}
\newtheorem{conjecture}[theorem]{Conjecture}

\let\P\p
\DeclareMathOperator{\GL}{GL}
\newcommand{\Hom}{{\mathrm{Hom}}}
\DeclareMathOperator{\Spec}{Spec}

\def\CC{\mathbb{C}}
\let\ot\otimes
\DeclareMathOperator{\rk}{rk}\let\rank\rk
\DeclareMathOperator{\im}{im}
\DeclareMathOperator{\id}{id}
\DeclareMathOperator{\sgn}{sgn}

\DeclareMathOperator{\Seg}{Seg}
\let\epsilon\varepsilon
\def\floor#1{\left\lfloor#1\right\rfloor}
\DeclareMathOperator{\Hilb}{Hilb}
\DeclareMathOperator{\T}{T}
\let\surjto\twoheadrightarrow
\DeclareMathOperator{\bcr}{bcr}
\def\koszul{\Phi_{\mathrm K}}
\def\tangency{\Phi_{\mathrm T}}\let\tang\tangency


\title[Nonlinear methods for tensors]{Nonlinear methods for tensors:\\ determinantal equations for secant varieties\\ beyond cactus}

\author{Matěj Doležálek}
\address{Fachbereich Mathematik und Statistik, Universität Konstanz, Konstanz, Germany}
\email{matej.dolezalek@uni-konstanz.de}

\author{Mateusz Micha\l ek}
\address{Fachbereich Mathematik und Statistik, Universität Konstanz, Konstanz, Germany}
\email{mateusz.michalek@uni-konstanz.de}

\begin{document}
\selectlanguage{english}

\begin{abstract}
We present a family of flattening methods of tensors which we call \emph{Kronecker-Koszul flattenings}, generalizing the famous Koszul flattenings and further equations of secant varieties studied among others by Landsberg, Manivel, Ottaviani and Strassen. We establish new border rank criteria given by vanishing of minors of Kronecker-Koszul flattenings.
We obtain the first explicit polynomial equations --  \emph{tangency flattenings} -- vanishing on secant varieties of Segre variety, but not vanishing on cactus varieties. Additionally, our polynomials have simple determinantal expressions.
As another application, we provide a new, computer-free proof that the border rank of the $2\times2$ matrix multiplication tensor is $7$.
\end{abstract}

\maketitle
\vspace{-1.2em}
\begin{flushright}
\small{\itshape
``Do not try and bend the spoon. That's impossible.\\ Instead, only try to realize the truth. There is no spoon. \\ Then you'll see that it is not the spoon that bends, it is only yourself.''}\par\smallskip
-- The Matrix
\end{flushright}

\section{Introduction}
\label{sec:intro}
The central problem concerning tensors is that of \emph{tensor decomposition} -- presenting a tensor as a sum of simple tensors, or computing the smallest possible size of a decomposition, called the \emph{tensor rank}. While the problem is NP-hard in general \cite{RankNPcomplete, hillar2013most}, providing estimates for rank or related notions such as border rank or asymptotic rank has been a fruitful field. Perhaps the most famous application is the study of computational complexity of matrix multiplication (e.g. \cite{strassen1969gaussian, coppersmith1982asymptotic, alman2024refined, williams2024new}). Further, a recent line of research relates low asymptotic rank of certain special tensors to faster than known algorithms for NP-hard problems \cite{bjorklund2025fast, bjorklund2024asymptotic, pratt2024stronger, MaMu-SotA}.
For an introduction to tensors from an algebraic perspective, we refer the reader to \cite[Chapter 9.2]{michalek2021invitation}, or to \cite{landsberg2011tensors} for a more extensive exposition. For the connections to complexity theory, we refer to \cite{burgisser2013algebraic, landsberg2017geometry}.

The main topic of this article is obtaining lower bounds on the border rank of tensors, a longstanding question in tensor geometry. The classical approach to the problem is by providing equations for secant varieties to the Segre variety. Flattenings, or maps which produce a matrix from a given tensor, have been a bountiful source of such lower bounds (e.g. \cite{langsberg-ottaviani-equationsforsecant, landsberg2008generalizations, landsberg-ottaviani-newbounds, landsberg-michalek-mamu, landsberg-michalek-haystack}), with the equations thus obtained coming in the form of minors. In particular, \emph{Strassen's equations}, further studied by Landsberg and Manivel and  expanded by Landsberg and Ottaviani to \emph{Koszul} and \emph{Young flattenings}, became the gold standard in the field. Apart from flattenings, the techniques of \emph{border apolarity} (see \cite{Buczynscy-apolarity}) and \emph{border substitution} (see \cite{landsberg2017geometry}) have been introduced and used to prove several new lower bounds in recent years. However, all of these methods have quite restrictive limits of application.

Indeed, current lower bound methods fall far short of being able to prove a tensor in $\CC^n\otimes\CC^n\otimes\CC^n$ to have the generic border rank, which is approximately $\frac{n^2}3$. Instead, known flattening methods give maximal lower bounds of up to $(2-\epsilon)n$ for small $\epsilon>0$. Combining algebraic equations with other methods, it is possible to provide examples of tensors of border rank $2.02n$ \cite{landsberg-michalek-haystack}. Still, there are well understood barriers which flattening methods cannot beat \cite{jarek_cactus, Galazka-bundles, EGOW18, garg2019more}. In contrast, the applications to asymptotic rank in computer science would require superlinear lower bounds on border rank.

The disparity between the generic border rank and available lower bounds has been attributed to the inclusion of secant varieties in \emph{cactus varieties} (see Section~\ref{sec:preliminaries}). In particular, known flattening methods have been unable to distinguish secant from cactus varieties, even when the inclusion is known to be strict  -- quoting from \cite{Conner_Harper_Landsberg_2023}:
\begin{quote}
``The geometric interpretation of the border rank lower bound barriers of \cite{EGOW18, Galazka-bundles} is that all equations obtained by taking minors, called rank methods, are actually equations for a larger variety than $\sigma_r(\Seg(\P A\times\P B\times\P C))$, called the $r$-th cactus variety \cite{Buczynscy-apolarity}.''
\end{quote}
Indeed, as proved in \cite{jarek_cactus, Galazka-bundles, galazka-thesis} all of the known methods producing equations for the $r$-th secant variety are equations for the $r$-th cactus variety. Further, it is known that the $r$-th cactus variety of the Segre variety in $(\CC^n)^{\ot 3}$ fills the ambient space for $r$ growing linearly in $n$ \cite{ballico2019note, bernardi2018polynomials, galkazka2023multigraded}.

For these reasons, the existence of easily computable determinantal equations vanishing on secant varieties but not vanishing on cactus varieties, has long been considered unlikely or even impossible in the tensor-geometric community.
In this article, we illustrate that it is not the determinantal expression which creates the cactus barrier, but rather the linear embedding of the tensor product into a matrix space.
Thus the starting point of our approach in this article is focusing on \emph{nonlinear} maps from tensor spaces to matrix spaces.
To obtain bounds on matrix rank with linear embeddings, it is enough to investigate the embedded image of the Segre variety and appeal to linearity. With nonlinear maps on the other hand, we need to work more generally, but this does yield stronger results. We obtain a new method, which we dub \emph{Kronecker-Koszul flattenings}, or more generally \emph{Kronecker-Young flattenings}, that allows one to produce
explicit and efficiently computable determinantal equations of secant varieties that do not generally vanish on cactus varieties. In particular, Kronecker-Koszul flattenings require less computation time and memory than methods based on a more direct identification of ideals defining the secant varieties (cf.~\cite{Hauenstein-Ikenmeyer-Landsberg}).

 Broadly speaking, our method starts by tensoring together several copies of the same tensor. This, depending on the interpretation of the codomain, may be viewed as a Kronecker power or Veronese embedding \cite{kaski2025universal}, and thus is the universal map for fixed degree polynomials. Next, we further tensor with several $2$-tensors corresponding to identity maps on various vector spaces.  Finally, we group and contract some of the factors to exterior powers. This results in what we call a \emph{Kronecker-Koszul tensor}. A Kronecker-Koszul flattening is then a classical flattening of the Kronecker-Koszul tensor. Note that due to initially taking a tensor power, Kronecker-Koszul tensors are in general nonlinear as functions of the starting tensor.
See Section~\ref{sec:construction} for details of the construction. There we also comment on what we call {Kronecker-Young flattenings}, a further generalization which synthesizes the approach of Young flattenings into the method.

Our main result uses the \emph{tangency flattening}, a very particular Kronecker-Koszul flattening which depends quadratically on a tensor. See the paragraph preceding Corollary~\ref{cor:tan_flat} for the definition.

\begin{theorem}
    \label{thrm:beat-cactus-introversion}
    For all $n\geq14$, minors of size $n(n-1)(n-2)+1$ of the tangency flattenings vanish on the $n$-th secant variety of the Segre variety in $\P(\CC^n\otimes\CC^n\otimes\CC^n)$ but do not vanish on its $n$-th cactus variety.
\end{theorem}
For example, for $n=14$ the minors in question are polynomials of degree $4370$ in $2744$ variables. Yet, they are explicit and we can determine their vanishing or nonvanishing very fast.
We will prove Theorem~\ref{thrm:beat-cactus-introversion} by lower-bounding border rank via rank of the tangency flattening in Corollary~\ref{cor:tan_flat} and exhibiting in Theorem~\ref{thrm:beat-cactus} a sequence of tensors of cactus rank $n$ on which the above-named minors of the tangency flattening do not vanish. The name ``tangency'' flattening is motivated by a conjectural connection to tangent spaces of Hilbert schemes of points -- see Conjecture~\ref{conj:tangency}. In Subsection~\ref{subsec:mamu}, we additionally show that tangency flattenings provide a new and completely elementary proof of the fact that $2\times2$ matrix multiplication tensor has border rank $7$.

While Kronecker-Koszul flattenings generalize Koszul flattenings (cf. \cite{langsberg-ottaviani-equationsforsecant, landsberg2008generalizations, landsberg-michalek-mamu, landsberg-michalek-haystack}), we also analyze in Section~\ref{sec:koszul} bounds that may be extracted via Koszul flattenings for the border rank (and border cactus rank) of structure tensors of certain algebras. It is known \cite{blaser2016degeneration} that a (finite) algebra is smoothable if and only if the associated structure tensor has minimal border rank. This is also one of the main challenges in application of the border apolarity method.
Further, two of the classical flattenings of structure tensors of algebras give rise to spaces of commuting matrices. Thus, many of the flattening methods cannot provide nontrivial bounds on border rank of such tensors. However, we show that it is possible to obtain lower bounds on border cactus rank of structure tensors of algebras strictly greater than the degree of the algebra, even using Koszul flattenings (see Proposition~\ref{prop:1de-koszul}).

Finally, we show that under mild assumptions, Koszul flattenings may be replaced and improved upon by certain quadratic Kronecker-Koszul flattenings for the purposes of lower-bounding border rank of tensor (see Proposition~\ref{prop:quadratic-better-than-koszul}).

\section*{Acknowledgments}
We are grateful to Jaros{\l}aw Buczy{\'n}ski, Fulvio Gesmundo, Joachim Jelisiejew, Petteri Kaski, J.~M.~Landsberg and Giorgio Ottaviani for interesting discussions, comments and encouragement.

This work has been supported by European Union’s HORIZON–MSCA-2023-DN-JD programme under the Horizon Europe
(HORIZON) Marie Skłodowska-Curie Actions, grant agreement 101120296 (TENORS).

\section{Preliminaries \& notation}
\label{sec:preliminaries}

For an integer $n$ we let $[n]=\{1,\dots,n\}$.
All vector spaces will be over complex numbers $\CC$. Given a basis $v_1,\dots,v_n$ of a vector space $V$, we denote $v_1^*,\dots,v_n^*$ the dual basis of $V^*$. The tensor product will also be taken over $\CC$.
Any binomial $\binom{n}{m}$ is equal to zero for $m>n$. Similarly $\bigwedge^m V=0$ for $m>\dim V$. All algebras we consider in this article will be finite, commutative, unital algebras over $\CC$. When referring to their dimension, we will mean dimension as a vector space, not the Krull dimension.

For a projective variety $X\subset \P(V)$ its \emph{$k$-th secant variety} $\sigma_k(X)\subset \P(V)$ is the Zariski closure of the union of projective spans of $k$-tuples of points of $X$. The \emph{$X$-rank} of $[v]\in \P(V)$ is the minimal number $k$ of points $[x_1],\dots,[x_k]\in X$, such that $v=\sum_{i=1}^k x_i$. The \emph{$k$-th cactus variety} $\kappa_k(X)\subset \P(V)$ is the Zariski closure of the union of projective spans of degree $k$ zero-dimensional subschemes of $X$. We always have $\sigma_k(X)\subset \kappa_k(X)$.  It is known that, after Veronese reembeding, nonsmoothable Gorenstein schemes contribute to cactus variety being larger than secant variety. So far the only explicit study of cactus variety is based on a good understanding of nonsmoothable Gorenstein schemes of degree $14$ in $\CC^6$ \cite{galazka_mandziuk_rupniewski_distinguishing}. 
For more details we refer to \cite{buczynska2014secant, iarrobino1999power}.

In our case we will usually take $X$ to be the Segre variety $X=\prod_{i=1}^r \P(V_i)\subset \P\left(\bigotimes_{i=1}^r V_i\right)$. In this case if $[T]\in \sigma_k(X)$ (resp.~$[T]\in \kappa_k(X)$), we say that $T$ has \emph{border rank} (resp.~\emph{border cactus rank}) at most $k$. The $X$-rank is simply called the \emph{rank} of the tensor. In case $r=2$ by identifying $V_1\ot V_2=\Hom(V_1^*, V_2)$ it coincides with the rank of a linear map.

Given a tensor $T\in \bigotimes_{i\in[r]} V_i$ and a subset $I\subset [r]$, we may identify $T$ with a linear map
\[T^I:\bigotimes_{i\in I} V_i^*\rightarrow \bigotimes_{i\in [r]\setminus I} V_i,\]
that is the \emph{classical flattening} of $T$. Ranks of classical flattenings bound the border rank of $T$ from below. For a tensor $T\in A\ot B \ot C$, where $\dim B=\dim C$, we say that $T$ is \emph{$1_A$-generic} if there exists $l\in A^*$, such that $T(l)\in B\ot C$ is a full rank matrix.

Recall that given any collection $F\subset \CC[x_1,\dots,x_n]$ of polynomials, we obtain the \emph{apolar ideal} $F^\perp\subset \CC[\partial_1,\dots,\partial_n]=D$ of operators annihilating each element of $F$. The \emph{apolar algebra} to $F$ is simply $D/F^\perp$. We will always be assuming that $F$ is finite, or equivalently a finitely generated $D$-module. The apolarity construction goes back to Macaulay \cite{Mac94}, for more details we refer to \cite[p.~XX]{iarrobino1999power}, \cite[p.~75]{dolgachev2012classical}. and \cite[Section 2.1]{flavi2025symmetric}.

The \emph{Hilbert scheme of $n$ points of $\CC^d$}, denoted $\Hilb_n\CC^d$, parameterizes surjective ring homomorphisms $S:=\CC[x_1,\dots,x_d]\surjto R$ to $n$-dimensional algebras $R$, or informally, $d$-generated $n$-dimensional algebras. Denoting the kernel of $S\surjto R$ as $I$, the tangent space at $R$ as a point on $\Hilb_n\CC^d$ may be identified as $\T_{[R]}\Hilb_n\CC^d = \Hom_S(I, R)$. A distinguished \emph{smoothable component} of $\Hilb_n\CC^d$ consists of smoothable algebras, and by considering the dense subset of smooth algebras, this component is $dn$-dimensional. For further information on Hilbert schemes of points, we refer the reader to the survey \cite{joachim-survey}.

\section{Kronecker-Koszul flattenings}
\label{sec:construction}

We consider tensors $T\in\bigotimes_{i=1}^r V_i$, where $d_i:=\dim V_i$. For our main construction we fix:
\begin{itemize}
\item a positive integer $k$;
\item $r$ partitions $\{\lambda_{i,1},\dots,\lambda_{i,{s_i}}\}_{1\leq i\leq r}$ of the set $[k]$, i.e.~for every $1\leq i\leq r$ we have disjoint, nonempty subsets $\lambda_{i,j}\subset [k]$, whose union is $[k]$;
\item for each $\lambda_{i,j}$ a nonnegative integer $d_{i,j}'$.
\end{itemize}
The objects we introduce below depend implicitly or explicitly on this fixed data. Before we proceed, let us  provide intuitive meaning of the objects above. The number $k$ will denote the tensor power of the tensor $T$. The $d_{i,j}'$ will provide exterior powers with which we will take the tensor product, similar to the construction of Koszul flattenings. The partitions $\lambda_{i,j}$ will tell us how to contract the tensor power of $T$ with the exterior powers, obtaining higher exterior powers.

We start by providing formalism on the exterior powers.
We fix a tensor
\[
J\in \bigotimes_{i=1}^r\bigotimes_{j=1}^{s_i}  \bigwedge^{d_{i,j}'}V_i\otimes \bigwedge^{d_{i,j}'} V_i^* ,
\]
which corresponds to the identity map on $\bigotimes_{i=1}^r\bigotimes_{j=1}^{s_i} \bigwedge^{d_{i,j}'}V_i$.
Here when $d_{i,j}'=0$, we may simply omit the term in the tensor product. This formally means tensoring with the one-dimensional space $\CC=\CC\otimes \CC=\bigwedge^0 V_i\otimes \bigwedge^0 V_i^*$. We will also consider each exterior power as one vector space, and refer to the corresponding tensor rank and border rank.\footnote{Most of our methods could possibly be upgraded by looking at $X$-ranks stemming from (products of) Grassmannians. However, we will not pursue this in the present article.}

Next, we consider the tensor product
$\tilde T:=T^{\otimes k}\otimes J$. The function $T\mapsto \tilde T$ is a degree $k$ polynomial map. By submultiplicativity of rank (resp.~border rank) under tensor product, it maps tensors of rank (resp.~border rank) at most $q$ to tensors of rank (resp.~border rank) at most
\[
(\rk T)^k\cdot\rk J = q^k\cdot \prod_{i=1}^r \prod_{j=1}^{s_i} \binom{d_{i}}{d_{i,j}'}.
\]
However, given any decomposition of $T$ as a sum of $q$ simple tensors, the induced decomposition for $T^{\otimes k}$ needs to be of a special form.
This is the property we are going to exploit.

We note that $T^{\otimes k}\in \bigotimes_{i=1}^r V_i^{\otimes k}$. However, the partitions $\lambda_{i,j}$ distinguish an ordering, thus we will write $V_i^{\otimes k}=V_{i,1}\otimes V_{i,2}\otimes\dots\otimes V_{i,k}$ where we have $V_{i,t}= V_i$. In this setting, we have
\[\tilde T\in \bigotimes_{i=1}^r \left(\left(\bigotimes_{t=1}^k V_{i,t} \right)\otimes \left(\bigotimes_{j=1}^{s_i}\bigwedge^{d_{i,j}'}V_i\otimes\bigwedge^{d_{i,j}'}V_{i}^*\right)\right)=\bigotimes_{i=1}^r \left(\left(\bigotimes_{j=1}^{s_i} \bigotimes_{t\in \lambda_{i,j}}V_{i,t} \right)\otimes \left(\bigotimes_{j=1}^{s_i}\bigwedge^{d_{i,j}'}V_i\otimes\bigwedge^{d_{i,j}'}V_{i}^*\right)\right).\]
The next step is to contract with the canonical maps
\[ \left(\bigotimes_{t\in \lambda_{i,j}}V_{i,t}\right)\otimes\bigwedge^{d_{i,j}'}V_i\rightarrow\bigwedge^{d_{i,j}'+|\lambda_{i,j}|}V_i.\]
We obtain a map
\[\pi:\bigotimes_{i=1}^r \left(\left(\bigotimes_{j=1}^{s_i} \bigotimes_{t\in \lambda_{i,j}}V_{i,t} \right)\otimes \left(\bigotimes_{j=1}^{s_i}\bigwedge^{d_{i,j}'}V_i\otimes\bigwedge^{d_{i,j}'}V_{i}^*\right)\right)\rightarrow \bigotimes_{i=1}^r\bigotimes_{j=1}^{s_i}\bigwedge^{d_{i,j}'+|\lambda_{i,j}|}V_i\otimes \bigwedge^{d_{i,j}'}V_i^*.\]

Our main object of interest is the tensor $T_\pi:=\pi(\tilde T)$, which we call the \emph{Kronecker-Koszul tensor}. We consider $T_\pi$ as a tensor in the tensor product of $\sum_{i=1}^r 2s_i$ many vector spaces, some of which may be one-dimensional, or even zero-dimensional in which case $T_\pi=0$. In other words, we consider each $\bigwedge^{d_{i,j}'+|\lambda_{i,j}|}V_i$ and each $\bigwedge^{d_{i,j}'}V_i^*$ as a single vector space and we compute rank and border rank of $T_\pi$ with respect to such a tensor product.

In order to state our main theorem bounding the (border) rank of $T_\pi$, we need to introduce one more object. A $k$-tuple $(l_1,\dots,l_k)\in [q]^k$ is called \emph{compatible} if the following condition holds:
\begin{itemize}
\item for any $1\leq t_1\neq t_2\leq k$, if  there exists $\lambda_{i,j}$ containing both $t_1$ and $t_2$, then $l_{t_1}\neq l_{t_2}$.
\end{itemize}
Let $C([q]^k)$ be the set of compatible $k$-tuples and $F(q):=|C([q]^k)|$ be the number of compatible $k$-tuples.
\begin{remark}
    \label{rmrk:chromatic}
    Consider a graph $G$ with vertex set $[k]$. For each $\lambda_{i,j}$ we connect the vertices in $\lambda_{i,j}$ pairwise with edges. Thus $G$ is a (usually not disjoint) union of cliques. Then $F(q)$ is the number of proper vertex colorings of $G$ with $q$ colors, i.e.~the evaluation of the chromatic polynomial $\chi_G(q)$.
\end{remark}
\begin{example}
Consider the Kronecker-Koszul construction with the following data: $r=3$ vector spaces, tensor power $k=4$, partitions with $\lambda_{1,1}=\{1,2,3\}$, $\lambda_{2,1}=\{3,4\}$ and all other $\lambda_{i,j}$ singletons. The graph $G$ of Remark~\ref{rmrk:chromatic} is then as in Figure~\ref{fig:graph}; in particular, recall that the parameters $d'_{i,j}$ as well as the vector space dimensions $d_i$ are irrelevant for the construction of $G$. A standard computation of $\chi_G$ then yields $F(q) = \chi_G(q) = q(q-1)^2(q-2)$.
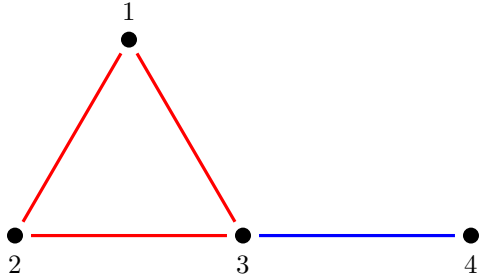
\begin{figure}[ht]
\centering
\begin{tikzpicture}[
    scale=3,
    every path/.style={shorten >=6pt, shorten <=6pt, line width=1.2pt}
    ]
  \coordinate (1) at (120:1cm);
  \coordinate (2) at (-1,0);
  \coordinate (3) at (0,0);
  \coordinate (4) at (1,0);
  \foreach\x in {2,...,4} \fill (\x) circle (1pt) node[below=3.5pt]{\x};
  \fill (1) circle (1pt) node[above=3.5pt]{1};
  \draw[red] (1) -- (2);
  \draw[red] (2) -- (3);
  \draw[red] (3) -- (1);
  \draw[blue] (3) -- (4);
\end{tikzpicture}
\caption{
    Graph $G$ obtained as the union of the clique $\color{red}\lambda_{1,1}=\{1,2,3\}$ and the clique $\color{blue}\lambda_{2,1}=\{3,4\}$. All other sets in the partitions are singletons, contributing no edges.
}
\label{fig:graph}
\end{figure}
\end{example}

\begin{theorem}\label{thm:main}
    If $T$ has rank (resp.~border rank) at most $q$, then $T_{\pi}$ has rank (resp.~border rank) at most $F(q)\cdot \prod_{i=1}^r\prod_{j=1}^{s_i} \binom{d_i-|\lambda_{i,j}|}{d_{i,j}'}$.
\end{theorem}
\begin{proof}
    As all the maps we apply are continuous, it is enough to prove the statement about rank. Thus we bound the rank of $T_{\pi}$, under the assumption that $T$ has rank at most $q$, i.e.~$T=\sum_{l=1}^q v_{l,1}\otimes\dots\otimes v_{l,r}$.
        We note that
    \[T^{\otimes k}\otimes J=\sum_{(l_1,\dots,l_k)\in[q]^k}\left(\bigotimes_{t=1}^kv_{l_t,1}\right)\otimes\dots\otimes\left(\bigotimes_{t=1}^kv_{l_t,r}\right)\otimes J.\]
    Hence,
    \[T_\pi=\pi(T^{\otimes k}\otimes J)=\sum_{(l_1,\dots,l_k)\in[q]^k}\pi\left(\left(\bigotimes_{t=1}^kv_{l_t,1}\right)\otimes\dots\otimes\left(\bigotimes_{t=1}^kv_{l_t,r}\right)\otimes J\right).\]
    We note that if $(l_1,\dots,l_k)$ is not compatible, then
\[\pi\left(\left(\bigotimes_{t=1}^kv_{l_t,1}\right)\otimes\dots\otimes\left(\bigotimes_{t=1}^kv_{l_t,r}\right)\otimes J\right)=0\]
as one will take the exterior product of two vectors that are the same. Thus we have
\[T_\pi=\sum_{(l_1,\dots,l_k)\in C([q]^k)}\pi\left(\left(\bigotimes_{t=1}^kv_{l_t,1}\right)\otimes\dots\otimes\left(\bigotimes_{t=1}^kv_{l_t,r}\right)\otimes J\right).\]
This sum has $F(q)$ many terms, and we will upper bound the rank of each term.
Fix $(l_1,\dots,l_k)$. If for any $\lambda_{i,j}$ the vectors $v_{l_t,i}$ for $t\in \lambda_{i,j}$ are not linearly independent, then
\[\pi\left(\left(\bigotimes_{t=1}^kv_{l_t,1}\right)\otimes\dots\otimes\left(\bigotimes_{t=1}^kv_{l_t,r}\right)\otimes J\right)=0.\]
Thus, we assume this is not the case and for every $\lambda_{i,j}$ we extend the sequence $S_{i,j}:=(v_{l_t,i})_{t\in\lambda_{i,j}}$ to a basis $B_{i,j}$ of $V_i$. Observe that $|S_{i,j}|=|\lambda_{i,j}|$.
We may represent $J$ as
\[J=\bigotimes_{i=1}^r\bigotimes_{j=1}^{s_i} \sum_{B\subset B_{i,j}}\bigwedge_{b\in B} b\otimes \bigwedge_{b\in B}b^*,\]
  where $ \sum_{B\subset B_{i,j}}$ is over all subsets of cardinality $d_{i,j}'$. We note that if $B\cap S_{i,j}\neq\emptyset$, then
  \[\pi\left(\left(\bigotimes_{t=1}^kv_{l_t,1}\right)\otimes\dots\otimes\left(\bigotimes_{t=1}^kv_{l_t,r}\right)\otimes \bigotimes_{i=1}^r\bigotimes_{j=1}^{s_i} \bigwedge_{b\in B} b\otimes \bigwedge_{b\in B}b^*\right)=0.\]
  Thus, we may change $\sum_{B\subset B_{i,j}}$ to summation over subsets of cardinality $d_{i,j}'$ that are disjoint from $S_{i,j}$. For fixed $(l_1,\dots,l_k)$ we obtain:
  \[\pi\left(\left(\bigotimes_{t=1}^kv_{l_t,1}\right)\otimes\dots\otimes\left(\bigotimes_{t=1}^kv_{l_t,r}\right)\otimes \bigotimes_{i=1}^r\bigotimes_{j=1}^{s_i} \sum_{B\subset B_{i,j}\setminus S_{i,j}}\bigwedge_{b\in B} b\otimes \bigwedge_{b\in B}b^* \right).\]
  This is a sum of $\prod_{i=1}^r\prod_{j=1}^{s_i}\binom{d_i-|\lambda_{i,j}|}{d_{i,j}'}$ rank one tensors. As we have $F(q)$ many possibilities for $(l_1,\dots,l_k)$ the theorem follows.
\end{proof}

\begin{definition}
    A \emph{Kronecker-Koszul flattening} of $T$ is any classical flattening of the tensor $T_\pi$.
\end{definition}

\begin{corollary}\label{cor:flat bound}
    If $T$ has border rank at most $q$, then the rank of a Kronecker-Koszul flattening of $T$ is at most $F(q)\cdot \prod_{i=1}^r\prod_{j=1}^{s_i} \binom{d_i-|\lambda_{i,j}|}{d_{i,j}'}$. In particular, minors of the Kronecker-Koszul flattening provide equations for secant varieties of the Segre variety.
\end{corollary}
\begin{proof}
    Follows from Theorem \ref{thm:main}, as ranks of classical flattenings lower bound border rank of the tensor.
\end{proof}
The bounds in Theorem \ref{thm:main} and Corollary \ref{cor:flat bound} do not have to be tight. This can happen for various reasons, some more opaque than others:

\begin{example}
    \label{ex:wedging-too-much}
    Consider $r=3$, $k=2$, $s_1=s_2=s_3=1$, $\lambda_{1,1}=\lambda_{2,1}=\lambda_{3,1}=\{1,2\}$ and all $d'_{i,j}$ zero. Denoting $T=\sum_{l=1}^q a_l\otimes b_l\otimes c_l \in A\otimes B\otimes C$, the corresponding Kronecker-Koszul tensor is then
    \[
        T_\pi = \sum_{1\leq i,j\leq q} (a_i\wedge a_j)\otimes(b_i\wedge b_j)\otimes(c_i\wedge c_j) \in \bigwedge^2 A\otimes \bigwedge^2 B\otimes \bigwedge^2 C.
    \]
    However, exchanging the roles of $i$ and $j$ in the summation and applying antisymmetry yields $T_\pi = -T_\pi$, hence $T_\pi=0$. In particular, the bound of Theorem~\ref{thm:main} is not tight.

    Experts may notice that this is due to the fact that the decomposition of $S^2(A\otimes B\otimes C)$ into irreducible $(\GL(A)\times\GL(B)\times\GL(C))$-representations does not contain $\bigwedge^2 A\otimes \bigwedge^2 B\otimes\bigwedge^2 C$.
\end{example}

\begin{example}
    \label{ex:mixed-koszul}
    Consider $r=3$, $d_1=d_2=d_3=q$, $k=1$ with $d'_{1,1}=d'_{2,1}=1$, $d'_{3,1}=0$. When $T = \sum_{i=1}^q a_i\otimes b_i\otimes c_i\in A\otimes B\otimes C$ is a concise tensor of minimal rank, Corollary~\ref{cor:flat bound} predicts the flattening
    \[
        \bigwedge^2 A^*\otimes \bigwedge^2 B^* \to C\otimes A^*\otimes B^*
    \]
    to have rank at most $q(q-1)^2$, based on the decomposition
    \[
        T_\pi = \sum_{\substack{1\leq i,j_1,j_2\leq n\\ j_1\neq i\neq j_2}} (a_i\wedge a_{j_1})\otimes (b_i\wedge b_{j_2})\otimes c_i\otimes a_{j_1}^* \otimes b_{j_2}^*.
    \]
    However, with the chosen flattening, for each term with $(i,j_1,j_2)=(x,y,y)$, we may combine it with $(i,j_1,j_2)=(y,x,x)$ into
    \[
        (a_x\wedge a_y)\otimes (b_x\wedge b_y) \otimes \left(c_x\otimes a_y^*\otimes b_y^* + c_y\otimes a_x^*\otimes b_x^*\right),
    \]
    which results in the Kronecker-Koszul flattening having rank at most $q(q-1)^2-\binom q2$. In particular, the bound of Corollary~\ref{cor:flat bound} is not tight.
\end{example}

However, in the case of many Kronecker-Koszul flattenings of interest the bound is tight. Below we present many such cases, which in particular cover all Kronecker-Koszul flattenings that we will use for our applications in Sections~\ref{sec:koszul} and \ref{sec:applications}.
\begin{proposition}
    Assume $d_1=\dots=d_r=q$ and for every $1\leq t\leq k$ there exist at least two $\lambda_{i_1(t),j_1(t)}$, $\lambda_{i_2(t),j_2(t)}= \{t\}$ such that $d'_{i_1(t),j_1(t)}=d'_{i_2(t),j_2(t)}=0$. If the codomain of $\pi$ is nonzero, i.e.~$d_{i,j}'+|\lambda_{i,j}|\leq d_i$, then there exists a tensor $T$ of rank $q$ and a Kronecker-Koszul flattening achieving the bound of Corollary \ref{cor:flat bound}. In particular, the bound in Theorem \ref{thm:main} is tight.
\end{proposition}
In particular, the assumptions always hold when $r=3$ and we use only one exterior power.
\begin{proof}
    Let $T$ be the unit tensor $T=\sum_{l=1}^q \otimes_{i=1}^r e_{i,l}$, where the respective $e_{i,l}$ form a basis of each $V_i$. We have $T_\pi\in \bigotimes_{i=1}^r\bigotimes_{j=1}^{s_i}\bigwedge^{d_{i,j}'+|\lambda_{i,j}|}V_i\otimes \bigwedge^{d_{i,j}'}V_i^*$. We choose the flattening of $T_\pi$ with domain \[\left(\bigotimes_{i=1}^r\bigotimes_{j=1}^{s_i} \bigwedge^{d_{i,j}'}V_i\right)
    \otimes \left(\bigotimes_{t=1}^k \bigwedge^{d'_{i_1(t),j_1(t)}+|\lambda_{i_1(t),j_1(t)}|}V_{i_1(t)}^*\right)=\left(\bigotimes_{i=1}^r\bigotimes_{j=1}^{s_i} \bigwedge^{d_{i,j}'}V_i\right)
    \otimes \left( \bigotimes_{t=1}^k V_{i_1(t)}^*\right).\]
    The bound of Theorem~\ref{thm:main} is based on a decomposition that is indexed by compatible tuples $(l_1,\dots,l_k)$ and subsequently $d'_{i,j}$ element subsets of a $(d_{i,j}-|\lambda_{i,j}|)$-element set. Working in the bases of exterior powers induced by the bases on $V_i$'s and with a fixed $(l_1,\dots,l_k)$, we denote $L_{i,j} := \{l_t\mid t\in\lambda_{i,j}\}$ and associate to $(l_1,\dots,l_k)$  and each subset $P_{i,j}\subset [q]\setminus L_{i,j}$ of cardinality $d'_{i,j}$ the basis vector
    \[
    \left( \bigotimes_{i=1}^r\bigotimes_{j=1}^{s_i} \bigwedge_{p\in P_{i,j}}e_{i,p}\right)\otimes \bigotimes_{t=1}^k e_{i_1(t),l_t}^*
    \]
    in the domain. Its image in the flattening will then be
    \[
        \bigotimes_{\substack{
            i\in [r],\ j\in[s_i],\\
            \text{$(i,j)\neq(i_1(t),j_1(t))$ for any $t$}
        }}
        \bigwedge_{p\in P_{i,j}\cup L_{i,j}} e_{i,p}.
    \]
    In particular, this is a basis vector in the codomain. If we now show that one can reconstruct from this image the original data $(l_1,\dots,l_k)$ and $P_{i,j}$, the proposition will follow.

    For this, we first look at the factor indexed by $(i,j)=(i_2(t), j_2(t))$ for every $t$. By our assumptions, this factor is just $V_{i_2(t)}$ and the basis element is $e_{i_2(t),l_t}$. Hence we reconstruct each $l_t$. Thus we know $L_{i,j}$ and may thereafter recover from each factor $P_{i,j} = (P_{i,j}\cup L_{i,j})\setminus L_{i,j}$, finishing the proof.
\end{proof}

We now present two special cases of our general construction, the first of which is well-known.
For $r=3$, $k=1$, $\lambda_{i,j}=\{1\}$ and only one $d_{i,j}'$ is nonzero we obtain the famous Koszul flattenings \cite{langsberg-ottaviani-equationsforsecant, landsberg2008generalizations}.

Consider $r=3$, $k=2$, $\lambda_{1,1}=\{1,2\}$, $\lambda_{2,1}=\lambda_{3,1}=\{1\}$, $\lambda_{2,2}=\lambda_{3,2}=\{2\}$, $d'_{1,1}=1$ and all other $d_{i,j}'$ zero. We consider a special Kronecker-Koszul flattening in this case given by
\[
\bigwedge^3 V_1^* \otimes V_2^*\otimes V_3^* \to V_1^*\otimes V_2\otimes V_3.
\]
We call it the \emph{tangency flattening} and denote it $\tang^{V_1}(T)$.
The following corollary is a consequence of Theorem \ref{thm:main}.
\begin{corollary}\label{cor:tan_flat}
    If $\dim V_1=\dim V_2=\dim V_3=n$, then $q(q-1)(n-2)+1$ minors of the tangency flattening provide equations for border rank $q$.
\end{corollary}
\begin{proof}
    As $\lambda_{1,1}=\{1,2\}$, compatibility for $(l_1,l_2)$ simply means $l_1\neq l_2$. Thus $F(q)=q(q-1)$, which corresponds to first two factors in each point. Further:
    \[\prod_{i=1}^3 \prod_{j=1}^{s_i}\binom{n-|\lambda_{i,j}|}{d_{i,j}'}=\binom{n-|\lambda_{1,1}|}{d_{1,1}'}=\binom{n-2}{1}=n-2.\]
    This provides the last factor.
\end{proof}

\subsection{Kronecker-Young flattenings}
\label{subsec:kronecker-young}
    In \cite{langsberg-ottaviani-equationsforsecant}, \emph{Young flattenings}, a generalization of Koszul flattenings using other irreducible representations of $\GL(V)$ in place of just $\bigwedge^{d'} V$, are employed. Our construction of Kronecker-Koszul flattenings directly generalizes to Kronecker-Young flattenings, by both projecting to and applying Schur functors different from exterior product.

    In this section we present one explicit construction that we find particularly interesting. By \cite{kaski2025universal} the linear span of image of the $k$-th Kronecker power of tensors in $A\ot B \ot C$ is $S^k(A\ot B \ot C)$. As a $\GL(A)\times \GL(B)\times \GL(C)$ representation, the latter space has a decomposition into isotypic components
    \[S^k(A\ot B \ot C)=\bigoplus_{\lambda,\mu,\rho\vdash k}\left(S^\lambda(A)\ot S^\mu(B)\ot S^\rho(C)\right)^{\bigoplus K_{\lambda,\mu,\rho}},\]
    where $K_{\lambda,\mu,\rho}$ is the Kronecker coefficient. In particular, when $\lambda=\mu=1^k$ and $\rho=(k)$ we have $K_{\lambda,\mu,\rho}=1$. This gives a canonical projection
    \[\pi:S^k(A\ot B \ot C)\rightarrow \bigwedge^k A\ot \bigwedge^k B\ot S^k(C).\]
    An example of Kronecker-Young flattening is the classical flattening of $\pi(T^{\otimes k})$ given by
    \[\Phi_{\mathrm M,k}^C(T):\bigwedge^kA^*\ot\bigwedge^kB^*\rightarrow S^k(C).\]
    The advantage of these choices is that the image of the map $\Phi_{\mathrm M,k}^C(T)$ has a direct interpretation in terms of the flattening $T^C:C^*\rightarrow A\ot B$. Namely, the image of $\Phi_{\mathrm M,k}^C(T)$ is the linear span of all $k\times k$ minors of $T^C$ viewed as a matrix with entries from $C$. For this reason, we call $\Phi_{\mathrm M,k}$ the \emph{$k$-minor flattening}.
\begin{lemma}
    If $\dim A=\dim B=\dim C=n$ and a tensor $T\in A\ot B\ot C$ has border rank at most $n$, then the rank of the $k$-minor flattening $\Phi_{\mathrm M,k}^C(T)$ is at most $\binom{n}{k}$.
\end{lemma}
\begin{proof}
    The $n$-th secant variety of the Segre variety is the closure of the $\GL(A)\times \GL(B)\times \GL(C)$ orbit of the unit tensor $I_n:=\sum_{i=1}^n a_i\ot b_i\ot c_i$, where $a_i$, $b_i$, $c_i$ form respectively bases of $A$, $B$, $C$. As our constructions are continuous and equivariant, it is enough to prove that $\rk \Phi_{\mathrm M,k}^C(I_n)=\binom{n}{k} $. We note that $I_n^C(C^*)$ is the space of diagonal matrices. Clearly, the linear span of $k\times k$ minors of the diagonal matrices is the linear span of squarefree monomials of degree $k$ in $n$ variables and thus has dimension $\binom{n}{k}$.
\end{proof}

The $k$-minor flattening $\Phi_{\mathrm M,k}^C$ has a few striking benefits:
\begin{enumerate}
    \item Its rank may be computed fast, even for $n$ and $k$ relatively large.
    \item It has a potential of proving nonsmoothability of Gorenstein schemes belonging to components of the Hilbert scheme with general tangent space larger than dimension of the smoothable component.
\end{enumerate}
Explicit examples justifying the second point are presented in Subsection \ref{subsec:minors-and-algs}.
    We justify the first point below. Given a tensor $T$, in order to lower-bound the dimension of the image of $\Phi_{\mathrm M,k}^C(T)$ by $r$, we do not have to take formal linear combinations of $\binom{n}{k}^2$ possible minors. Instead it is possible to:
    \begin{enumerate}
        \item Choose $r$ general points $p_i\in C^*$ and $r$ general pairs of matrices $U_j\in \Hom(\CC^k, A)$, $V_j\in\Hom(\CC^k, B)$.
        \item Form an $r\times r$ matrix $(m_{ij})$ by taking $m_{ij}$ to be the determinant of the $k\times k$ matrix $U_j^T T^C(p_i) V_j$.
        \item Compute the rank of $(m_{ij})$.
    \end{enumerate}
Clearly, if the method above returns a matrix of rank $r'$, then $\rk \Phi_{\mathrm M,k}^C(T)\geq r'$. Further, if $r\geq \rk\Phi_{\mathrm M,k}^C(T)$, then the method returns $r' = \rk\Phi_{\mathrm M,k}^C(T)$ with probability $1$.

\section{Koszul flattenings}
\label{sec:koszul}

Throughout this section, let us specialize to $3$-tensors, denoting the three starting vector spaces as $A$, $B$ and $C$. Whenever $T\in A\otimes B\otimes C$ is a tensor, we use $\koszul^{A,B}(T)$ to denote the first Koszul flattening
\begin{align*}
    \bigwedge^2 A^* \otimes B^* &\to A^*\otimes C,\\
    (\alpha\wedge\alpha')\otimes \beta &\mapsto \alpha\otimes T(\alpha',\beta) - \alpha'\otimes T(\alpha,\beta).
\end{align*}
In bases, it can be described as a matrix with a particular blocking: letting $n:=\dim A$, choosing a basis $\alpha_1,\dots,\alpha_n$ of $A^*$ and $M_i:= T(\alpha_i)\in B\otimes C$, then $\koszul^{A,B}(T)$ is the block matrix
\begin{equation}
    \label{eq:koszul-blocks}
    \bordermatrix{
        ~ & \{1,2\} & \dots & \{i,j\} & \dots & \{n-1,n\} \cr
        1 & M_2 & & \cdots & & 0\cr
        2 & -M_1 & & & & 0\cr
        \vdots & & \ddots\cr
        i & \vdots & & \epsilon_{i,j} M_j & & \cdots\cr
        \vdots & & & & \ddots\cr
        n-1 & 0 & & & & M_{n-1}\cr
        n & 0 & & \cdots & & -M_{n-1}  \crcr
    },
\end{equation}
where $\epsilon_{i,j} := 1$ if $i<j$ and $-1$ if $i>j$. We define the five other flattenings $\koszul^{A,C}$, $\koszul^{B,A}$, $\koszul^{B,C}$, $\koszul^{C,A}$ and $\koszul^{C,B}$ analogously.

\subsection{Structure tensors of algebras}
\label{subsec:koszul-and-algebras}

Let us adopt the convention that structure tensors of commutative algebras $R$ are always considered as elements of $R^*\otimes R^*\otimes R$, i.e. the partial symmetry is between the first and second factors.

Due to this partial symmetry, the structure tensor $T$ essentially has only three Koszul flattenings: $\koszul^{A,B}(T)$, $\koszul^{A,C}(T)$ and $\koszul^{C,A}(T)$. The following lemma is known -- it is similar to interpretations of Strassen's commutativity equations \cite{strassen1983rank} due to Landsberg, Manivel and Ottaviani \cite{landsberg2008generalizations, ottaviani2008symplectic}, \cite[Section 2.1]{landsberg2017abelian}.
\begin{lemma}\label{lem:easy rank Koszul for algebras}
    For any structure tensor $T$ of a commutative algebra $R$ the Koszul flattenings $\koszul^{A,B}(T)$ and $\koszul^{A,C}(T)$ both have rank $n(n-1)$.
\end{lemma}
\begin{proof}
    As $R$ is commutative, matrices in the space $T(R)\subset R^*\otimes R\simeq \Hom(R,R)$ commute with one another. Further, as $1\in R$ we have $T(1)=\id_R\in T(R)$. Thus, in \eqref{eq:koszul-blocks} we may take $M_1=\id_R$. We may use these $n-1$ full-rank blocks $M_1$ to eliminate the other blocks using row and column operations. In general, this creates blocks of the form $M_iM_j-M_jM_i$ in other positions. In our case, all such blocks will vanish due to the commuting property of $T(R)$. Thus, the ranks of $\koszul^{A,B}(T)$ and $\koszul^{A,C}(T)$ are exactly $(\rk M_1)(n-1)=n(n-1)$.
\end{proof}

The previous lemma implies that  the flattenings $\koszul^{A,B}(T)$ and $\koszul^{A,C}(T)$ may never be used to prove nonsmoothability of the algebra. Further, the same argument holds even after restricting to a subspace $V\subset R$ and arguing for the tensor $T'\in V^*\otimes R^*\otimes R$. This often led to the wrong impression that Koszul flattenings can never be used to prove nonsmoothability of algebras. Indeed, when $R$ is Gorenstein, the tensor $T$ is symmetric \cite[Theorem p.~1507]{huang2020vanishing} and similar conclusions can be obtained for different flattenings.
Surprisingly however, in general, the other Koszul flattening $\koszul^{C,A}(T)$ can yield nontrivial results:

\begin{proposition}
    \label{prop:1de-koszul}
    Let $R$ be the apolar algebra to $e\leq \binom{d+1}2$ general quadrics in $d$ variables. If $e\geq3$, then the structure tensor $T$ of $R$ has border cactus rank at least $\floor{\frac32d}+3$. In particular, if $3\leq e\leq \floor{\frac d2}+1$, then $R$ is nonsmoothable.
\end{proposition}
\begin{proof}
    By the apolar construction, $R$ is a graded algebra with Hilbert function $(1,d,e)$ and hence dimension $n := 1+d+e$.
    Let us choose a general three-dimensional subspace $V\subset R^*$ and consider the tensor $T'\in R^*\otimes R^*\otimes V^*$ and its Koszul flattening $\koszul^{C,A}(T')$. This amounts to choosing general matrices $M_i \in T(R^*)\subset R^*\otimes R^*$, $i=1,2,3$ and building the matrix
    \[
        K := \begin{pmatrix}
            M_2 & M_3 &\\
            -M_1 & & M_3\\
            & -M_1 & -M_2
        \end{pmatrix}.
    \]
    We will show that $\rk K = 6+2\floor{\frac32d}$. Denoting border cactus rank $\bcr$, \cite[Theorem 1.18]{galazka-thesis}, \cite{jarek_cactus} will then imply $\bcr(T)\geq\bcr(T')\geq 3+\floor{\frac32d}$, since any rank $1$ tensor $Z\in R^*\otimes R^*\otimes V^*$ will get $\rk\koszul^{C,A}(Z)=2$.

    A matrix $T(\gamma)\in T(R^*)\subset R^*\otimes R^*$ has the form
    \[
        \begin{pmatrix}
            x_0 & \mathbf x_1^T & \mathbf x_2^T\\
            \mathbf x_1 & S(\mathbf x_2) & \\
            \mathbf x_2
        \end{pmatrix},
    \]
    where $x_0$, $\mathbf x_1$ and $\mathbf x_2$ are (vectors of) coordinates describing the linear form $\gamma$ respectively on degree $0$, $1$ and $2$ parts of $R$, and $S(\mathbf x_2)$ is a $d\times d$ symmetric matrix linearly dependent on $\mathbf x_2$. Picking three such matrices for $M_i$, we obtain
    \[
        K = \begin{pmatrix}
            y_0 & \mathbf y_1^T & \mathbf y_2^T & z_0 & \mathbf z_1^T & \mathbf z_2^T & & & \\
            \mathbf y_1 & S(\mathbf y_2) & & \mathbf z_1 & S(\mathbf z_2) & & & &\\
            \mathbf y_2 & & & \mathbf z_2\\
            -x_0 & -\mathbf x_1^T & -\mathbf x_2^T & & & & z_0 &\mathbf z_1^T & \mathbf z_2^2\\
            -\mathbf x_1 & -S(\mathbf x_2) & & & & & \mathbf z_1 & S(\mathbf z_2) & \\
            -\mathbf x_2 & & & & & & \mathbf z_2 \\
            &&& -x_0 & -\mathbf x_1^T & -\mathbf x_2^T & -y_0 & -\mathbf y_1^T & -\mathbf y_2^T\\
            &&& -\mathbf x_1 & -S(\mathbf x_2) & & -\mathbf y_1 & -S(\mathbf y_2) &\\
            &&& -\mathbf x_2 & & & -\mathbf y_2
        \end{pmatrix}.
    \]
    First, we observe that the initial rows and initial columns of each block altogether contribute $6$ to the rank -- the vectors $\mathbf x_2$, $\mathbf y_2$ and $\mathbf z_2$ are chosen generally from an $e$-dimensional space, so with $e\geq 3$ they are linearly independent. After using these parts of the blocking to eliminate the rest of their respective rows and columns, we denote $U:=S(\mathbf x_2)$, $V:=S(\mathbf y_2)$, $W:=S(\mathbf z_2)$ and are the left to consider
    \[
        \begin{pmatrix}
            V & W & \\
            -U & & W\\
            & -U & -V
        \end{pmatrix}.
    \]

    Since the algebra $R$ was chosen generally and $U$, $V$, $W$ are chosen generally in $T(R^*)$, these are altogether just three general symmetric $d\times d$ matrices; in particular, $U$ is of full rank. By decomposing $U=PP^T$ and replacing $V$, $W$ with $P^{-1}V(P^{-1})^T$, $P^{-1}W(P^{-1})^T$ respectively, we may assume $U$ is the identity matrix. Next, by row and column operations we transform
    \[
        \begin{pmatrix}
            V & W & \\
            -\id & & W\\
            & -\id & -V
        \end{pmatrix}
        \qquad\text{into}\qquad
        \begin{pmatrix}
            & & VW-WV\\
            -\id & &\\
            & -\id &
        \end{pmatrix}.
    \]
    Thus it remains to show that for general $d\times d$ symmetric matrices $V$, $W$, the commutator $VW-WV$ has rank $2\floor{\frac d2}$. Firstly, this commutator will be antisymmetric, hence it can have rank at most $2\floor{\frac d2}$. By semicontinuity of rank, it then suffices to show that there exists a pair of symmetric matrices $V$, $W$ for which $\rk(VW-WV) = 2\floor{\frac d2}$. This is easily verified e.g. for
    \[
        V = \begin{pmatrix}
            &&1\\
            &\iddots&\\
            1&&
        \end{pmatrix}
        \qquad\text{and}\qquad
        W = \underbrace{
            \hbox spread-60pt{$\displaystyle
            \begin{pmatrix}
                1&&&&&\\
                &\ddots&&&&\\
                &&1&&&\\
                &&&0&&\\
                &&&&\ddots&\\
                &&&&&0
            \end{pmatrix}.
            $\hss}
        }_{\text{$\floor{\frac d2}$ ones}}
        \qedhere
    \]
\end{proof}
The study of smoothability of algebras of type $(1,d,e)$ is a classical topic \cite{shafarevich1990deformations}.
We note that for $d=4$ and $e=3$ we obtain the classical result about the reducibility of the Hilbert scheme of $8$ points \cite{iarrobino1978some, cartwright2009hilbert}. For other results on components of Hilbert schemes we refer to \cite{satriano2023small, jelisiejew2019elementary}. However, to our knowledge the bounds on the border cactus rank beyond the dimension of the algebra are new. 
Further, Proposition \ref{prop:1de-koszul} is not exhaustive -- $\koszul^{C,A}$ may prove nonsmoothability in other cases as well, though the block structure and the subsequent rank analysis become more complicated as the Hilbert function acquires more nonzero entries.
\begin{example}
    \label{ex:1463}
    Let $q_1,q_2,q_3\in\CC[x_1,x_2]$ be two general cubics and $l_{i,1},l_{i,2}\in\CC[y_1,y_2,y_3,y_4]$, $i=1,2,3$ six general linear forms. Then we set $f_i:=q_i(l_{i,1},l_{i,2})\in \CC[y_1,\dots,y_4]$ and let $R$ be the apolar algebra to $\{f_1,f_2,f_3\}$. Its Hilbert function is $(1,4,6,3)$, its dimension $14$, but a direct computer-assisted computation with its structure tensor $T$ yields\footnote{As we are making general choices, formally we prove $\rank \koszul^{C,A}(T) \geq 184$ and equality holds for general $T$ with probability one.} $\rank \koszul^{C,A}(T) = 184 > 14\cdot13$, hence $R$ is nonsmoothable.
\end{example}

\subsection{Comparison of Koszul flattenings with a quadratic flattening}
\label{subsec:koszul-and-quadratic}

Let us consider a quadratic Kronecker-Koszul flattening stemming from $k=2$, a single nontrivial partition $\lambda_{1,1} = \{1,2\}$ and all $d'_{i,j}=0$. Let us denote it
\begin{align*}
    \Phi_{\mathrm Q}^A(T): \bigwedge^2 A^* \otimes B^* \otimes C^* &\to C\otimes B,\\
    (\alpha\wedge\alpha')\otimes\beta\otimes\gamma &\mapsto T(\alpha,\beta)\otimes T(\alpha',\gamma) - T(\alpha',\beta)\otimes T(\alpha,\gamma).
\end{align*}
We define $\Phi_{\mathrm Q}^B$ and $\Phi_{\mathrm Q}^C$ analogously.

\begin{proposition}
    \label{prop:quadratic-better-than-koszul}
    The space $\im\Phi_{\mathrm Q}^C(T)\subset A\otimes B$ equals the sum of all $(T(\alpha)\otimes \id_A)\left(\im\koszul^{C,B}(T)\right)$ over $\alpha\in A^*$, where we consider $T(\alpha)$ as a map $C^*\to B$. In particular, if $\dim B=\dim C$ and $T$ is $1_A$-generic, this implies $\rank \Phi_{\mathrm Q}^C(T) \geq \rank\koszul^{C,B}(T)$.
\end{proposition}
\begin{proof}
    Evaluating on any element $\alpha\otimes(\gamma\wedge\gamma')\otimes\beta$, it is easily verified that the composition
    \[\begin{tikzcd}
        A^*\otimes \bigwedge^2 C^* \otimes B^* \arrow[rrr, "\id_{A^*}\otimes\koszul^{C,B}(T)"] &&& A^*\otimes C^*\otimes A \arrow[rr, "T^B\otimes \id_A"] && B\otimes A
    \end{tikzcd}\]
    (where the second arrow uses the classical flattening $T^B: A^*\otimes C^*\to B$) equals $\Phi_{\mathrm Q}^C$ up to reordering of factors. The image of this composition is clearly the sum of all $T(\alpha)\left(\im\koszul^{C,B}(T)\right)$, which proves the first part of the proposition. When $T$ is $1_A$-generic, we take an $\alpha\in A^*$ such that $T(\alpha)$ is of full rank. Then $T(\alpha)\otimes\id_{A}$ is injective, hence at least one summand contributing to $\im\Phi_{\mathrm Q}^C(T)$ is as large as $\im \koszul^{C,B}(T)$, proving the inequality.
\end{proof}

The proposition means that under a genericity assumption, $\Phi_{\mathrm Q}^C$ is ``at least as useful'' as $\koszul^{C,B}$ for attempting to prove nonminimal border rank of a tensor $T\in \CC^n\otimes\CC^n\otimes\CC^n$, since for both of these flattenings Corollary~\ref{cor:flat bound} gives $n(n-1)$ as the bound on the dimension of the image
to be overcome. However, this still leaves open the possibility that $\Phi_{\mathrm Q}^C$ might not prove nonsmoothability, but $\koszul^{C,B}$ does when considered after projecting to a smaller space similar to the proof of Proposition~\ref{prop:1de-koszul}.

There do exist cases where the inequality of Proposition~\ref{prop:quadratic-better-than-koszul} becomes strict, i.e.~the quadratic flattening strictly outperforms the corresponding Koszul flattening.
\begin{example}
    Let $T_1$ be a general rank $6$ tensor in $\CC^5\otimes\CC^5\otimes\CC^5$ and $T_2\in\CC^4\otimes\CC^4\otimes\CC^4$ the structure tensor of the algebra $\CC[x,y,z]/((x,y,z)^2)$. Then letting $T:=T_1\oplus T_2\oplus T_2\in \CC^{13}\otimes\CC^{13}\otimes\CC^{13}$, it may be verified by a direct calculation that $\rank \koszul^{C, B}(T) = 13\cdot 12$, but $\rank \Phi_{\mathrm Q}^C(T) = 13\cdot12+1$. Hence $\Phi_{\mathrm Q}^C$ proves nonminimal border rank, but $\koszul^{C,B}$ does not.
\end{example}

For structure tensors of commutative algebras however, we can prove an equality -- our quadratic flattening and the corresponding Koszul flattening give the same information:
\begin{proposition}
    \label{prop:koszul-equals-quad-on-algs}
    Suppose that $T\in R^*\otimes R^*\otimes R = A\otimes B\otimes C$ is the structure tensor of an $n$-dimensional commutative algebra $R$. Then
    \begin{equation}
        \label{eq:koszul-quaddratic-algebra}
        \rk \Phi_{\mathrm Q}^C(T) = \rk \koszul^{C,B}(T)
        \qquad\text{and}\qquad
        \rk \Phi_{\mathrm Q}^A(T) = \rk \koszul^{A,C}(T) = n(n-1).
    \end{equation}
\end{proposition}
\begin{proof}
    Let us prove only the second set of equalities of \eqref{eq:koszul-quaddratic-algebra}, since the proof of the first one is very similar.

    Explicitly, the two flattenings in question map
    \begin{align*}
        \Phi_{\mathrm Q}^A(T): \bigwedge^2 R \otimes R \otimes R^* &\to R \otimes R^*,\\
        (a\wedge a')\otimes b\otimes \gamma &\mapsto (ab)\otimes \gamma(a'\cdot\bullet) - (a'b)\otimes \gamma(a\cdot\bullet)
    \end{align*}
    and
    \begin{align*}
        \koszul^{A,C}(T): \bigwedge^2 R \otimes R^* &\to R \otimes R^*,\\
        (a\wedge a')\otimes\gamma &\mapsto a\otimes \gamma(a'\cdot\bullet) - a'\otimes \gamma(a\cdot\bullet),
    \end{align*}
    where we use $\gamma(a\cdot\bullet)$ to denote the linear form $(x\mapsto \gamma(ax))\in R^*$. We claim that in fact $\Phi_{\mathrm Q}^A(T)$ and $\koszul^{A,C}(T)$ have the same images. The inclusion $\im\Phi^{A,C}(T)\subseteq \im\Phi_{\mathrm Q}^A$ is obvious by considering $b=1$ in the latter map.

    For the opposite inclusion, note that invertible elements form a dense set in $R$, so in particular, $R$ may be spanned by them, hence $\im \Phi_{\mathrm Q}^A$ is spanned by images of $(a\wedge a')\otimes b\otimes\gamma$ with $b$ invertible. These lie in $\im \koszul^{A,C}$ by the following: we set
    \[
        \tilde a := ba, \qquad \tilde a' := ba', \qquad \tilde\gamma := \gamma(b^{-1}\cdot\bullet),
    \]
    which then results in
    \begin{multline*}
        \koszul^{A,C}(T)\left(
            (\tilde a\wedge \tilde a') \otimes \tilde \gamma
        \right)
        =
        \tilde a\otimes\tilde \gamma(\tilde a'\cdot\bullet) - \tilde a'\otimes\tilde\gamma(\tilde a\cdot\bullet)
        =
        (ab)\otimes\gamma(b^{-1}ba'\cdot\bullet) - (a'b)\otimes \gamma(b^{-1}ba\cdot\bullet)
        =\\=
        (ab)\otimes \gamma(a'\cdot\bullet)-(a'b)\otimes\gamma(a\cdot\bullet)
        =
        \Phi_{\mathrm Q}^A(T)\left(
            (a\wedge a') \otimes b\otimes\gamma
        \right).
    \end{multline*}
    This justifies $\rank \Phi_{\mathrm Q}^A(T) = \rank \koszul^{A,C}(T)$. That this common rank equals $n(n-1)$ then follows from Lemma \ref{lem:easy rank Koszul for algebras}.
\end{proof}

\section{Applications}
\label{sec:applications}

\subsection{Matrix multiplication}
\label{subsec:mamu}

It is well-known that the border rank of the $2\times 2$ matrix multiplication tensor $M_2$ equals $7$ \cite{Landsberg2006, Hauenstein-Ikenmeyer-Landsberg}. However, only recently the ``first hand-checkable algebraic proof'' appeared \cite{Conner_Harper_Landsberg_2023}. Below we show that tangency flattenings provide an easy proof.
\begin{proposition}\label{prop:M2=7}
    The border rank of $M_2$ equals seven.
\end{proposition}
\begin{proof}
Strassen \cite{strassen1969gaussian} provided an explicit rank seven decomposition of $M_2$. We focus on lower-bounding the border rank.

    By Corollary \ref{cor:tan_flat} for $n=4$, $q=6$ we have to show that the tangency flattening has rank greater than $60$. We will prove that it is $64$. We fix a basis $(i,j)$ of $\CC^4$, where $i,j\in\{0,1\}$ and we use $\neg i$ to denote the negation of $i$. Recall that $M_2=\sum_{i,j,k\in \{0,1\}} (i,j)\otimes (j,k)\otimes (k,i)$.

    We will prove that the tangency flattening is a block diagonal matrix with five blocks: a $48\times 48$ identity matrix and four isomorphic, nondegenerate $4\times 4$ matrices.

The tangency flattening of $M_2$ is the restriction to $\bigwedge^3\CC^4\otimes \CC^4\otimes \CC^4$ of
\begin{align}
\label{eq:mamu-tangency}
(\CC^4)^{\otimes 3}\otimes \CC^4\otimes \CC^4 &\to \CC^4\otimes \CC^4\otimes \CC^4,\\
\nonumber
(i_1,j_1)\otimes (i_2,j_2)\otimes (i_3,j_3)\otimes (i_4,j_4)\otimes (i_5,j_5) &\mapsto
\begin{cases}
(i_1,j_1)\otimes (i_2,j_4)\otimes (i_5,j_3), & \text{if $j_2=i_4$ and $i_3 = j_5$,}\\
0, & \text{otherwise.}
\end{cases}
\end{align}
We divide the basis vectors in the domain into five groups.

The first group is formed by vectors of the form $\bigl((\neg j_5,\neg i_4)\wedge(i_2,i_4)\wedge (j_5,j_3)\bigr)\otimes (i_4,j_4)\otimes (i_5,j_5)$. By first choosing $i_4,j_4,i_5,j_5$ (16 options), and then choosing exterior products that contain $(\neg j_5,\neg i_4)$ (3 options), we see there are $48$ such (nonzero) vectors.
Note that this already forces the exterior product to be of the form above, since the two remaining exterior factors are chosen from $(\neg j_5, i_4)$, $(j_5, \neg i_4)$ and $(j_5,i_4)$.
Although the exterior product expands as a signed sum of six tensor products, only one of those -- the one ordered as in the presentation above -- maps to a nonzero vector in \eqref{eq:mamu-tangency}. Thus, the image of each such basis vector is $(\neg j_5,\neg i_4)\otimes (i_2,j_4)\otimes (i_5,j_3)$. Note that such images must satisfy $(\neg j_5,\neg i_4)\neq (\neg i_2,\neg j_3)$.

The next four groups are indexed by $j_4,i_5\in \{0,1\}$. For fixed $j_4,i_5$ the group consists of the four basis vectors $\bigl((j_5,i_4)\wedge (\neg j_5,i_4)\wedge (j_5,\neg i_4)\bigr)\otimes (i_4,j_4)\otimes (i_5,j_5)$, where $i_4, j_5$ are arbitrary.
The image of each such vector is:
\[(j_5,i_4)\otimes (\neg j_5,j_4)\otimes (i_5,\neg i_4)-(\neg j_5,i_4)\otimes (j_5,j_4)\otimes (i_5,\neg i_4)-(j_5,\neg i_4)\otimes (\neg j_5,j_4)\otimes (i_5,i_4).\]
In particular, the images are linear combinations of basis vectors that are distinct from images of the first group as well as other groups -- in each term, the first tensor factor consist exactly of  negations of the first coordinate of the second factor  and the second coordinate of the third factor. Meanwhile, the group is determined by reading off $j_4$ and $i_5$ from the second coordinate of the second factor and first coordinate of the third factor, respectively.

Thus indeed the tangency flattening is in block-diagonal form. It remains for us to prove that the isomorphic $4\times 4$ blocks are of full rank.
Each such matrix has the following form, where we label $(j_5,i_4)\otimes (\neg j_5,j_4)\otimes (i_5,\neg i_4)$ with the first factor $(j_5,i_4)$ and subsequently order rows and columns by $(j_5,i_4)=(0,0)$, $(0,1)$, $(1,0)$, $(1,1)$:
\[\begin{pmatrix}
    1 &  -1&-1&0\\
    -1& 1&0&-1\\
    -1& 0&1&-1\\
    0& -1&-1&1\\
\end{pmatrix}.\]
The determinant of such a matrix is $-3$, thus indeed the matrix is of full rank, and so is the tangency flattening.
\end{proof}
\begin{remark}
    After we shared the first version of the preprint, J. M. Landsberg observed that the proof of Proposition \ref{prop:M2=7} may be even further simplified using representation theory. Indeed, recall that in the coordinate-free way, $M_2\in (V_1\ot V_2^*)\ot (V_2\ot V_3^*)\ot (V_1^*\ot V_3)$ is a $\GL(V_1)\times \GL(V_2)\times \GL(V_3)$ invariant tensor. Thus the codomain of the tangency flattening is a sum of $8$ irreducible representations. Hence, it is enough to check that explicit $8$ vectors are in the image. By analyzing the discrete part of the stabilizer of $M_2$ and its interaction with identity, one can exploit symmetry to reduce the number of vectors even more. Although for $M_2$ each proof is very simple, we believe that this remark may be helpful when working with larger matrix multiplication tensors. 
\end{remark}

\subsection{Distinguishing cactus and secant varieties}
\label{subsec:beat-cactus}

Before we proceed further, we prove how the tangency flattenings $\tangency$ behave under direct sum.
\begin{proposition}
    \label{prop:tangency-direct-sum}
    For $i=1,2$ let $T_i\in A_i\otimes B_i\otimes C_i$, where $\dim A_i=\dim B_i=\dim C_i$, and
    \[
        A:= A_1\oplus A_2, \qquad B:= B_1\oplus B_2, \qquad C:= C_1\oplus C_2, \qquad T:= T_1\oplus T_2\in A\otimes B\otimes C.
    \]
    Then
    \begin{align}
        \label{eq:tangency-direct-sum}
        \rank \tang^A(T) &= \rank \tang^A(T_1) + \rank \tang^A(T_2) + \dim A_2\cdot\rank\Phi_{\mathrm Q}^A(T_1) + \dim A_1\cdot\rank \Phi_{\mathrm Q}^A(T_2) +{}\\&\nonumber\quad{}+ \rk T_2^B\cdot\rk \koszul^{A,B}(T_1) + \rk T_2^C\cdot\rk\koszul^{A,C}(T_1) + \rk T_1^B\cdot\rk \koszul^{A,B}(T_2) + \rk T_1^C\cdot\koszul^{A,C}(T_2),
    \end{align}
    where $\Phi_{\mathrm Q}$ denotes the quadratic flattening of Subsection~\ref{subsec:koszul-and-quadratic} and $T_i^B: A_i^*\otimes C_i^*\to B_i$, $T_i^C: A_i^*\otimes B_i^*\to C_i$ are the classical flattenings.
\end{proposition}
\begin{proof}
    Via the canonical inclusions stemming from direct sums, we may view both $T_i$ as (nonconcise) tensors in $A\otimes B\otimes C$. Consider the identity matrices $\id_i\in A_i^*\otimes A_i \subset A^*\otimes A$. Then we may view $\tang^A(T)$ as the evaluation $\Phi(T_1+T_2, T_1+T_2, \id_1+\id_2)$ of the trilinear map
    \[
        \Phi: (A\otimes B\otimes C) \times (A\otimes B\otimes C) \times (A^*\otimes A) \to \Hom\left(
            \bigwedge^3 A^*\otimes B^* \otimes C^*,\ A^* \otimes B\otimes C
        \right)
    \]
    given by
    \[
        \Phi(F_1,F_2,F_3)\Bigl((\alpha^{(1)}\wedge\alpha^{(2)}\wedge\alpha^{(3)})\otimes\beta\otimes\gamma\Bigr) := \sum_{\sigma\in S_3}\sgn(\sigma) F_3(\alpha^{(\sigma(1))})\otimes F_1(\alpha^{(\sigma(2))}, \gamma) \otimes F_2(\alpha^{(\sigma(3))}, \beta).
    \]

    With
    \[
    \bigwedge^3 A^* = \bigwedge^3 A_1^* \oplus \left( \bigwedge^2 A_1^* \otimes A_2^*\right) \oplus \left( A_1^*\otimes\bigwedge^2 A_2^*\right) \oplus\bigwedge^3 A_2^*,
    \]
    we may rewrite the codomain of $\Phi$ as
    \[
        \bigoplus_{u,v,w,i,j,k}
        \Hom\left(
        \left(\bigwedge^u A_1^*\otimes \bigwedge^{3-u} A_2^*\right)
        \otimes B_v^* \otimes C_w^*
        ,\
        A_i^* \otimes B_j\otimes C_k
        \right).
    \]
    A priori, this would have $128$ summands, however, the projection of $\Phi(T_1+T_2, T_1+T_2, \id_1+\id_2)$ to most will be identically zero. Namely, notice that once $j\in\{1,2\}$ is chosen, we must choose $w=j$ in order to obtain a nonzero contribution to the image via $T_j^B: A_j^*\otimes C_j^*\to B_j$. Similarly we force $v=k$. Finally, for the $A^*$ factors, we must have the equality of multisets
    \[
        \{i,j,k\} = \{ \underbrace{1}_{\text{$u$-times}}, \underbrace{2}_{\text{$(3-u)$-times}} \},
    \]
    so altogether, a choice of $i,j,k\in\{1,2\}$ uniquely determines $u$, $v$, $w$. These eight contributions to $\Phi(T_1+T_2,\allowbreak T_1+T_2, \id_1+\id_2)$ will have images in independent spaces, so we merely need to sum the ranks of the eight projections. Let us list them in turn:
    \begin{itemize}
        \item For $(i,j,k)=(1,1,1)$, we recover the tangency flattening $\tangency^A(T_1)$. Analogously for $(i,j,k)=(2,2,2)$.
        \item For $(i,j,k)=(1,2,2)$, we obtain the Kronecker product of maps $\id_1:A_1^*\to A_1^*$ and $\bigwedge^2A_2^*\otimes B_2^*\otimes C_2^*\to B_2\otimes C_2$, which is just the quadratic flattening $\Phi_{\mathrm Q}^A(T_2)$. Analogously for $(i,j,k)=(2,1,1)$.
        \item For $(i,j,k)=(1,1,2)$, we obtain the Kronecker product of $\bigwedge^2 A_1^*\otimes C_1^*\to A_1^*\otimes B_1$, which is the Koszul flattening $\koszul^{A,C}(T_1)$, and $A_2^*\otimes B_2^*\to C_2$, which is just the classical flattening $T_2^C$. Analogously for $(i,j,k)=(1,2,1)$, $(2,2,1)$ and $(2,1,2)$.
    \end{itemize}
    Summing these ranks then establishes \eqref{eq:tangency-direct-sum}.
\end{proof}

\begin{corollary}
    \label{cor:alg-tangency-sum}
    If $T_i\in R_i^*\otimes R_i^*\otimes R_i$, $i=1,2$ are structure tensors of $n_i$-dimensional commutative algebras $R_i$, then the structure tensor $T_1\oplus T_2$ of the $n:=(n_1+n_2)$-dimensional algebra $R_1\oplus R_2$ satisfies
    \[
        \rank\tang^A(T_1\oplus T_2) - n(n-1)(n-2) = \Bigl(\rank \tang^A(T_1) - n_1(n_1-1)(n_1-2)\Bigr) + \Bigl(\rank\tang^A(T_2) - n_2(n_2-1)(n_2-2)\Bigr).
    \]
\end{corollary}
\begin{proof}
    Along with the conclusions of Proposition~\ref{prop:koszul-equals-quad-on-algs}, observe that structure tensors of algebras are concise. Thus all but the first two terms of the right-hand side of \eqref{eq:tangency-direct-sum} simplify, yielding
    \[
        \rank\tang^A(T_1\oplus T_2) = \rank \tang^A(T_1) + \rank \tang^A(T_2) + 3n_1n_2(n_1+n_2-2).
    \]
    This rearranges to the Corollary.
\end{proof}
The following lemma is well-known to experts. It can be derived e.g.~from the results in \cite{jelisiejew2024concise}. We include the proof for the sake of completeness.
\begin{lemma}\label{lem:GorCactus}
   Let $R$ be a finite Gorenstein algebra and $T\in (R^*)^{\otimes k}\ot R$ its structure tensor corresponding to the map $R^k\rightarrow R$, $(r_1,\dots,r_k)\mapsto r_1\cdots r_k$. The cactus rank of $T$ equals $d=\dim_{\CC} R$.
\end{lemma}
\begin{proof}
   As $T$ is concise it is enough to prove that it has cactus rank at most $d$. Thus we have to exhibit a scheme of length $d$, which will be $\Spec R$, inside the Segre variety, so that $T$ is in its linear
span.  As cactus rank is subadditive under direct sum we may assume $(R,\mathfrak m)$ is local.

We have a canonical surjection of the symmetric algebra over $m$ to $R$:
    \[S^\bullet(\mathfrak m)\twoheadrightarrow R,\text{ and hence }(S^\bullet(\mathfrak m))^{\otimes (k+1)}\twoheadrightarrow R.\]
It induces an inclusion $\Spec R\hookrightarrow \prod_{i=1}^{k+1} \mathfrak m^*$. There is a canonical linear function on $R^*$ given by $1\in R$. It allows us to identify $\mathfrak m^*$ with the affine piece
\[
\{l\in R^* \mid l(1)=1\} = \{[l]\in\P(R^*)\mid l(1)\neq 0\} \subset \P(R^*).
\]
We thus have a natural inclusion
\[\Spec R\hookrightarrow \prod_{i=1}^{k+1}\P(R^*)\subset \P\left((R^*)^{\otimes(k+1)}\right),\]
where the second inclusion is the Segre embedding.

A homogeneous linear form on $\P\left((R^*)^{\otimes(k+1)}\right)$ is an element of $\left(\P\left((R^*)^{\otimes(k+1)}\right)\right)^*= R^{\otimes(k+1)}$. Such a linear form $l=\sum r_1\otimes\dots\otimes r_{k+1}$ vanishes on the image of $\Spec R$ if and only if $\sum r_1\cdots r_{k+1}=0$.

So far we did not use the Gorenstein assumption. Recall that one of the equivalent definitions is the existence of a linear form $l_0\in R^*$, such that the pairing $R\times R\rightarrow\CC$, $(r_1,r_2)\mapsto l_0(r_1r_2)$ is perfect.
We use this pairing to induce the isomorphism $\P\left((R^*)^{\otimes k}\otimes R\right)\simeq\P\left((R^*)^{\otimes k}\right)$ given by $[l_1\otimes\dots\otimes  l_k\otimes r]\mapsto [l_1\otimes\dots\otimes l_k\otimes l(r\cdot \bullet)]$.

To finish the proof, we will show that $T\in \P\left((R^*)^{\otimes k}\otimes R\right)$ is in the projective span of the image of $\Spec R$. Set $r_1:=1$ and choose a basis $r_2,\dots,r_d$ of $\mathfrak m$. To present $T$, we introduce the constants $\lambda_{i_1,\dots,i_k,j}$ defined via
\begin{equation}
\label{eq:tensor-entries}
r_{i_1}\cdots r_{i_k}=:\sum_j \lambda_{i_1,\dots,i_k,j}r_j.
\end{equation}
Thus
\[T=\sum_{i_1,\dots,i_k,j} \lambda_{i_1,\dots,i_k,j}r_{i_1}^*\otimes\dots\otimes r_{i_k}^*\otimes r_j.\]
We now fix a linear form
\[
l=\sum_{i_1,\dots,i_k,i} \mu_{i_1,\dots,i_k,i}r_{i_1}\otimes\dots\otimes r_{i_k}\otimes l_0(r_i\cdot \bullet)
\]
on $\P\left((R^*)^{\otimes k}\otimes R\right)$. The vanishing of $l$ on the image of $\Spec R$ is equivalent to
\[
\sum_{i_1,\dots,i_k,i} \mu_{i_1,\dots,i_k,i}r_{i_1}\cdots r_{i_k}r_i=0.
\]
After substituting \eqref{eq:tensor-entries} we obtain
\[\sum_{i_1,\dots,i_k,i,j}\mu_{i_1,\dots,i_k,i}\lambda_{i_1,\dots,i_k,j}r_j r_i=0.\]
Applying $l_0$ then yields
\begin{equation}\label{eq:l_0}
    \sum_{i_1,\dots,i_k,i,j}\mu_{i_1,\dots,i_k,i}\lambda_{i_1,\dots,i_k,j}l_0(r_j r_i)=0.
\end{equation}
Our aim is to prove that $l(T)=0$, as then we may conclude that $T$ belongs to the linear span of the image of $\Spec R$. We have:
\begin{multline*}
    l(T)=\sum_{i_1,\dots,i_k,j} \lambda_{i_1,\dots,i_k,j}l(r_{i_1}^*\otimes\dots\otimes r_{i_k}^*\otimes r_j)=\\=\sum_{i_1,\dots,i_k,j} \lambda_{i_1,\dots,i_k,j}\sum_{i_1',\dots,i_k',i}\mu_{i_1',\dots,i_k',i}r_{i_1}^*(r_{i_1'})\otimes\dots\otimes r_{i_k}^*(r_{i_k'})\otimes l_0(r_jr_i)
    =\\=\sum_{i_1,\dots,i_k,j,i} \lambda_{i_1,\dots,i_k,j}\mu_{i_1,\dots,i_k,i}l_0(r_jr_i).
\end{multline*}
This is indeed zero by \eqref{eq:l_0}.
\end{proof}

\begin{theorem}
    \label{thrm:beat-cactus}
    The $(n(n-1)(n-2)+1)$-minors of the tangency flattenings for border rank $n$ in $\CC^n\otimes \CC^n\otimes \CC^n$ do not vanish on the $n$-th cactus variety for all $n\geq 14$.
\end{theorem}
\begin{proof}
By Lemma \ref{lem:GorCactus}, structure tensors of Gorenstein algebras have minimal cactus rank.
    Thus, it suffices to produce examples of structure tensors of $n$-dimensional Gorenstein algebras whose tangency flattenings have ranks strictly higher than $n(n-1)(n-2)$.

    For $n=14$, we use the apolar algebra to a general cubic in six variables, where we verify directly with computer assistance that its structure tensor has tangency flattening of $\rank \tang^A(T) = 2192 = 14\cdot13\cdot12+8$.

    Afterward, we proceed by induction: if $R$ is an $n$-dimensional Gorenstein algebra whose structure tensor $T$ has a tangency flattening of rank higher than $n(n-1)(n-2)$, we consider $R\oplus\CC$. The structure tensor of $\CC$ is the $1\times1\times1$ identity tensor $I_1$, whose tangency flattening is trivial due to $\dim \bigwedge^3\CC^1 = 0 = 1\cdot0\cdot(-1)$. Hence Corollary~\ref{cor:alg-tangency-sum} yields
    \begin{align*}
        \rank\tang^A(T\oplus I_1) - (n+1)n(n-1) &= \rank \tang^A(T) - n(n-1)(n-2) > 0.
        \qedhere
    \end{align*}
\end{proof}

\subsection{Minor flattenings and algebras}\label{subsec:minors-and-algs}

In this subsection we show that the minor flattening $\Phi_{\mathrm M,4}^C$ may prove nonsmoothability of local Gorenstein algebras $R$ even when the algebra belongs to a larger component of the Hilbert scheme than the smoothable component.
\begin{example}\label{ex:k-minor flattening}
    Letting $R_6$ and $R_8$ be the apolar algebras to a general cubic in $6$ and $8$ variables respectively, let $T_6$ and $T_8$ be their structure tensors. With computer assistance, we obtain\footnote{
        The ranks here are obtained with the method described in Subsection~\ref{subsec:kronecker-young}, hence formally, they are lower bounds unconditionally and equal the true ranks with probability $1$.
    }
    \begin{align*}
        \rk\Phi_{\mathrm M,3}^C(T_6) &= 372 > 364=\binom{14}3, &
        \rk\Phi_{\mathrm M,4}^C(T_8) &= 3123 > 3060=\binom{18}4,
    \end{align*}
    which proves that these algebras are nonsmoothable. Note that the algebra $R_8$ belongs to a component of the Hilbert scheme of dimension $155 > 144 = 8\cdot18$.
\end{example}

\section{Future directions}

\subsection{Hilbert schemes}

The following conjecture gives tangency flattenings their name:

\begin{conjecture}
    \label{conj:tangency}
    Let $\CC[x_1,\dots,x_d]\surjto R$ be a surjection to an $n$-dimensional algebra $R$, let $[R]$ be the corresponding point of $\Hilb_n\CC^d$ and let $T$ be the structure tensor of $R$. Then
    \[
        \rank\tang^C(T) - n(n-1)(n-2) = dn - \dim\T_{[R]}\Hilb_n\CC^d.
    \]
\end{conjecture}
Informally: the tangency flattening of the structure tensor of $R$ wins (or loses) over the rank threshold $n(n-1)(n-2)$ by exactly as much as how smaller (or larger) the tangent space at $[R]$ to $\Hilb_n\CC^d$ is compared to the smoothable component.

In view of the conjecture, our proof of Theorem~\ref{thrm:beat-cactus} may be seen as merely recovering the known fact that the apolar algebra to a general cubic in six variables has too small of a tangent space to the Hilbert scheme.
The algebras covered by Proposition~\ref{prop:1de-koszul} and Example~\ref{ex:1463} include both algebras with small tangent spaces and large tangent spaces (cf. the dimension counting of \cite{iarrobino-inventiones}, which even lower-bounds the dimension of the component of $\Hilb_n\CC^d$). Hence already Koszul flattenings are capable of proving nonsmoothability of algebras with large tangent spaces.

By investigating flattening behavior under direct sums, we can construct other algebras where a Kronecker-Koszul flattening succeeds despite a small rank of the tangency flattening:

\begin{example}[higher tangency flattening]
    Let $R_1$ be the apolar algebra to $3$ general quadrics in $4$ variables and $R_2$ the algebra corresponding to the smoothable scheme obtained by scaling a general smooth scheme $R_2'\subset \CC^6$ of degree $11$ by $t\in\CC^*$ and taking the limit at $t=0$. The dimensions of these algebras are $n_1=8$ and $n_2=11$ respectively, we then take $R:= R_1\oplus R_2$.

    Let $T_1$, $T_2$, $T=T_1\oplus T_2$ be the respective structure tensors. It may be directly computed that
    \[
        \rk\tangency^C(T_1) = 8\cdot7\cdot6 + 7 \qquad\text{and}\qquad \rk\tangency^C(T_2) = 11\cdot10\cdot9 - 9,
    \]
    hence by Corollary~\ref{cor:alg-tangency-sum}, $\rk\tangency^C(T) = 19\cdot18\cdot17 - 2$, meaning the tangency flattening does not prove nonsmoothability of $R$.

    Let us modify the tangency flattening by using an identity map on $\bigwedge^2 C^*$ instead of just $C^*$. We might call the following map
    \[
        \Phi_{\mathrm T,2}^C(T): \bigwedge^4 C^*\otimes A^*\otimes B^* \to \bigwedge^2 C^*\otimes A\otimes B
    \]
    the ``second tangency flattening''. Analogously to Proposition~\ref{prop:tangency-direct-sum} and Corollary~\ref{cor:alg-tangency-sum}, its behavior under direct sums may be analyzed to prove that on structure tensors of algebras, it satisfies a formula incorporating both the second and usual tangency flattenings:
    \begin{align*}
        &{}\rk\Phi_{\mathrm T,2}^C(T) - n(n-1)\binom{n-2}2 ={}\\{}&\quad{}=
        \left(\rk\Phi_{\mathrm T,2}^C(T_1) - n_1(n_1-1)\binom{n_1-2}2\right)
        +
        \left(\rk\Phi_{\mathrm T,2}^C(T_2) - n_2(n_2-1)\binom{n_2-2}2\right)+{}\\
        &\qquad{}+ 2n_2\cdot\left(\rk\Phi_{\mathrm T}^C(T_1) - n_1(n-1-1)(n_1-2)\right)
        + 2n_1\cdot\left(\rk\Phi_{\mathrm T}^C(T_2) - n_2(n-2-1)(n_2-2)\right).
    \end{align*}
    Evaluating explicitly $\rk\Phi_{\mathrm T,2}^C(T_1) = 8\cdot 7\cdot\binom62 + 28$ and $\rk\Phi_{\mathrm T,2}^C(T_2) = 11\cdot 10\cdot\binom92 - 36$ then yields
    \[
    \rk\Phi_{\mathrm T,2}^C(T) - 19\cdot18\cdot\binom{17}2 = 28 - 36 + 2\cdot 11\cdot 7 - 2\cdot 8\cdot 9 = 2 > 0,
    \]
    hence $\Phi_{\mathrm T,2}^C$ proves nonsmoothability of $R$ despite $\tangency^C$ not being able to.
\end{example}

Similarly to Conjecture \ref{conj:tangency} we may investigate Example \ref{ex:k-minor flattening}. We notice that $\rk\Phi_{\mathrm M,3}^C(T_6)$ exceeds $\binom{14}{3}$ by $372-364=8$. This is also the difference of the dimensions of tangent spaces to the Hilbert scheme respectively at a smooth scheme (of degree six) and at the algebra corresponding to $T_6$.
\begin{conjecture}
    Let $\CC[x_1,\dots,x_d]\surjto R$ be a surjection to an $n$-dimensional Gorenstein algebra $R$, let $[R]$ be the corresponding point of $\Hilb_n\CC^d$ and let $T$ be the structure tensor of $R$. Then
    \[
        \rank\Phi_{\mathrm M,3}^C(T) - \binom{n}{3} = dn - \dim\T_{[R]}\Hilb_n\CC^d.
    \]
\end{conjecture}

\subsection{Relations to the Bodensee program}
In \cite{dinu2025applications, conner2021characteristic, michalek2023enumerative} the Bodensee program was initiated. It suggests the study of tensors via geometry of the associated subvarieties of Grassmannians and varieties of complete collineations. The program is directly related to the study of matroids via subvarieties of the permutohedral variety \cite{JunePhD}.
In this direction, collineation varieties of tensors were introduced and studied in \cite{gesmundo2025collineation}. We note that the $k$-minor flattening $\Phi_{\mathrm M,k}^C$, whose rank may distinguish the cactus variety from the secant variety, has image equal to the linear span of the associated collineation variety. Further, \cite[Proposition 3.3]{jagiella2025characteristic} and \cite[Theorem 1.1]{bik2020jordan} imply that for structure tensors $T$ of $n$-dimensional local Gorenstein algebras we have $\rk \Phi_{\mathrm M,k}^C(T)=\rk\Phi_{\mathrm M,n-k}^C(T)$.

The dimension of the linear span of a variety is one of the most basic invariants one can imagine. The possible generalizations to other values of the Hilbert function are a very natural direction. It is also natural to ask which other algebraic, cohomological or geometric invariants of generalized Schubert varieties or collineation varieties could give control over the border rank of a tensor.

\subsection{Special tensors and equations of secant varieties}
Many special tensors and their border rank are of particular interest. Among them are for example matrix multiplication, monomials \cite{landsberg2010ranks, han2022new}, determinant tensor \cite{ju2025new, Conner_Harper_Landsberg_2023} or matching tensors \cite{bjorklund2024asymptotic, bjorklund2025fast}. We did not extensively check all our equations on these examples, noting that for $M_2$ our equations certainly provide a new approach. It is a natural next step to further investigate lower bonds provided by Kronecker-Young flattenings.

Another line of research studies equations of secant varieties from qualitative \cite{draisma2014bounded, sam2017ideals} or explicit \cite{raicu2012secant} perspectives. It is natural to ask about the extent to which Kronecker-Young flattenings describe secant varieties, as they are generalizations of Koszul flattenings.

\subsection{Strengthening the method}

In closing, let us outline several open-ended suggestions as to how the method of Kronecker-Koszul flattenings might potentially be strengthened.

\begin{question}[more complicated contractions]
    What if, in the construction of Section~\ref{sec:construction}, we contract several $\bigwedge^{d'_{i,j}} V_i^*$ together? Or what if we additionally contract several $\bigwedge^{d'_{i,j}+|\lambda_{i,j}|}V_i$?
\end{question}

\begin{question}[more inputs]
    Our methods create a large tensor from an input tensor $T$ and identity maps on some vectors spaces. The resulting rank criteria then essentially stem from the observation that as $T$ deforms to another tensor $T'$ via an action of the underlying endomorphism rings, the rank of the Kronecker-Koszul flattening can only decrease.

    The same would apply if we considered a version of Kronecker-Koszul flattenings with multiple input tensors -- when deforming any single one of the inputs via endomorphisms of the underlying spaces, the rank could only drop. Could stronger results on a tensor $T$ be extracted by choosing suitable ``witness tensors'' dependent on $T$ as the additional inputs?
\end{question}

\begin{question}[flattening the Kronecker-Koszul tensor nonclassically]
    In the construction of Section~\ref{sec:construction}, we create a Kronecker-Koszul tensor and then take a classical flattening of it. What if we instead finished with taking some nonclassical flattening, e.g. by bracketing the tensor as a $3$-tensor in some way and taking a Koszul flattening?
\end{question}
Of course, any Koszul flattening of the Kronecker-Koszul tensor would reduce to a direct sum of several Kronecker-Young flattenings via plethysms. One might hope however that viewing the construction through repeated exterior powers might be more tractable than handling a complicated representation-theoretic expression directly.

\bibliography{biblio}
\bibliographystyle{alpha_four}
\end{document}